\documentclass[a4paper, 12pt]{article}
\usepackage{amsmath, amsthm, amssymb,times}
\usepackage[all]{xy}
\usepackage{graphicx}

\numberwithin{equation}{section}

\theoremstyle{definition}
\newtheorem{theorem}{Theorem}[section]
\newtheorem{lemma}[theorem]{Lemma}
\newtheorem{cor}[theorem]{Corollary}
\newtheorem{prop}[theorem]{Proposition}
\newtheorem{fact}[theorem]{Fact}
\newtheorem{remark}[theorem]{Remark}
\newtheorem{definition}[theorem]{Definition}

\newcommand{\Q}{\mathbb Q}
\newcommand{\C}{\mathbb C}
\renewcommand{\S}{\mathcal{S}}
\newcommand{\GL}[1]{{\rm GL}{\left( #1 \right)}}
\newcommand{\SU}[1]{{\rm SU}{\left( #1 \right)}}
\newcommand{\ctr}[1]{{Z}( #1 )}
\newcommand{\Z}{{\mathbb Z}}
\newcommand{\Aut}[1]{{\rm Aut}\left( #1 \right)}
\newcommand{\Out}[1]{{\rm Out}\left( #1 \right)}
\newcommand{\Inn}[1]{{\rm Inn}\left( #1 \right)}
\newcommand{\End}[1]{{\rm End}\left( #1 \right)}
\newcommand{\Ad}[1]{{\rm Ad}{#1} }
\newcommand{\ap}{\mathcal{A}\phi}
\newcommand{\Hom}[1]{{\rm Hom}\left( #1 \right)}
\renewcommand{\ker}{\operatorname{Ker}}

\newcommand{\SL}{{\rm SL}}
\newcommand{\M}{\mathcal{M}}
\newcommand{\mcg}{\M_{g,*}}
\newcommand{\fg}{\pi_1(\Sigma_g, *)}
\newcommand{\HG}{\Hom{\fg, G}}

\renewcommand{\bold}[1]{\smallskip \noindent {\bf \boldmath #1 }\nopagebreak[4]}
\newcommand{\PM}{\operatorname{\mathcal{PM}}_{g,n}^{p}}
\newcommand{\R}{\mathbb{R}}
\newcommand{\PSL}[1]{\operatorname{PSL}(#1)}

\begin{document}
\title{On visualization of the linearity problem for   mapping class groups of surfaces}
\author{Yasushi Kasahara}
\date{}

\maketitle

\begin{abstract}
 We derive two types of linearity conditions for  mapping class groups of orientable surfaces:
 one for once-punctured surface, and the other for  closed surface, respectively.
 For the once-punctured case, the condition is described in terms of the action of 
 the mapping class  group  on the deformation space  of  linear representations of the 
 fundamental group of the corresponding
 closed surface. For the closed case, the condition is described in terms of the 
 vector space generated  by the isotopy classes of essential simple closed curves on the corresponding 
 surface. The latter condition also describes the linearity  for the mapping class group of
 compact orientable surface with  boundary, up to center.
\end{abstract}

\section{Introduction} \par

One of the fundamental problems on mapping class groups of surfaces is 
the {\em linearity problem}. It asks whether or not
the mapping class group in question admits a faithful {\em finite dimensional} linear representation
over some field. If a group admits a faithful finite dimensional linear representation over a 
particular field $K$, we will say that the group is {\em $K$-linear}.
\par

Let $\Sigma_g$ be an orientable closed  surface of genus $g$, and $\M_g$ its mapping class group.
By definition, $\M_g$ is the group of the isotopy classes of the orientation preserving 
homeomorphisms of $\Sigma_g$. 
We also consider the once-punctured surface $\Sigma_{g,*}$, 
by which we mean the pair of the surface $\Sigma_g$ and a distinguished marked point $* \in \Sigma_g$.
The mapping class group of $\Sigma_{g,*}$, denoted by $\mcg$, 
is defined as the group of the isotopy classes of the 
orientation preserving homeomorphisms which preserve the marked point where the isotopy is 
assumed to always fix the marked point.
As for the case of genus $g=1$, it is classically known that both $\M_1$ and $\M_{1,*}$ are isomorphic 
to $\SL(2, \Z)$, and hence are $\Q$-linear.  
The linearity of $\M_2$ was established  rather recently by Korkmaz \cite{korkmaz}
and Bigelow--Budney \cite{bigelow-budney},  using the celebrated works 
by Bigelow \cite{bigelow} and Krammer \cite{krammer} that Artin's braid groups are linear. 
However, the linearity of $\M_g$ for $g \geq 3$,  and also the linearity of $\mcg$ for $g \geq 2$, both
seem to remain open. 
\par \medskip

The purpose of this paper is to derive two types of conditions each of which is  
{\em equivalent} to the linearity of $\mcg$,  and that of $\M_g$, respectively. 
For the former case, in Section \ref{once-punctured}, we show that the 
$K$-linearity of $\mcg$ for $g \geq 2$ is equivalent to
the existence of a faithful $K$-linear representation of $\fg$ which represents 
a {\em global fixed point} of the action of $\M_g$ on the corresponding 
deformation space (Theorem \ref{ThmA}). We also combine this condition with the recent result by
Korkmaz \cite{korkmaz2} to observe that such a global fixed point does not exist in low dimensions
with $K=\C$ (Remark \ref{low_degrees}).
For the latter case, in Section \ref{closed}, we obtain a linearity condition for $\M_g$ with $g \geq 3$
in terms of the vector space generated by the isotopy classes of essential simple closed curves on $\Sigma_g$
(Corollary \ref{linearity_closed_case}). 
In the course of its proof, we also obtain a linearity condition {\em up to center} for the case of 
a compact connected orientable surface of genus $g \geq 1$ possibly with boundary 
(Theorem \ref{linearity_up_to_center}). 
\par

In our arguments to derive the linearity conditions, we need criteria for the faithfulness 
of a representation. These are given by Lemmas \ref{Lemma1} and \ref{injectivity_up_to_center}, 
respectively. We remark that these two criteria are based on a common observation: a homomorphism 
of groups is injective if it is injective on {\em a subset with trivial centralizer} 
which is invariant under conjugation.
\par\medskip
		 \bold{Conventions.} In this paper, we adopt the following two conventions.
		  The multiplication of two loops is read from left to right. Therefore, the multiplication 
		 $\gamma_1 \gamma_2$ of two loops $\gamma_1$ and $\gamma_2$ is the loop which first 
		 traverses $\gamma_1$ and then $\gamma_2$. 
		 The composition of two mappings is read from right to left. Hence for the 
		 multiplication $fh$ of two elements of the mapping class group, the class $h$ is applied 
		 first, and then $f$ is applied.
		 These conventions affect especially the description of the homomorphism $i$ in 
		 \eqref{Birman_exact}.
\par\medskip

\bold{Notation.} 
For a field $K$, $\End{n,K}$ denotes the $K$-vector space consisting of all the square 
matrices of degree $n$ over $K$. The group of the invertible matrices in $\End{n,K}$ is denoted
by $\GL{n,K}$. For a $K$-vector space $V$, $\GL{V}$ denotes the group of the $K$-linear isomorphisms of $V$.
For a group $G$, $\Aut{G}$ denotes the group of the automorphisms of $G$.
\par\medskip

\bold{Acknowledgements.}
The author is grateful to Louis Funar and Makoto Sakuma for valuable discussions and comments. 
He is grateful to Masatoshi Sato for an enlightening conversation. He is grateful to the referee for 
helpful comments.
The author was partially supported by the Grant-in-Aid
for Scientific Research (C) (No.23540102) from the Japan Society for Promotion of Sciences.
\par

\section{Once-punctured surface} \label{once-punctured} \par

Recall that $\mcg$ denotes the mapping class group of the once-punctured surface $\Sigma_{g,*}$.
In this section, we assume $g \geq 2$ and derive a condition equivalent to the linearity of $\mcg$.

\subsection{Preliminary}\par

The mapping class group $\mcg$ naturally  acts on the fundamental group $\fg$ of $\Sigma_g$ based at 
the marked point $*$, 
which gives rise to a homomorphism 
\begin{equation} \label{natural_action}
 (\cdot)_* : \mcg \to \Aut{\fg}.
\end{equation}
By a version of the classical Dehn--Nielsen theorem, this homomorphism is {\em injective} 
({\em c.f.} Farb--Margalit \cite{fm}).
\par\medskip

Forgetting the marked point $*$ induces 
a homomorphism of  $\mcg$ onto $\M_g$, the mapping class group of the {\em closed} surface 
$\Sigma_g$. The kernel of this homomorphism is isomorphic to $\fg$, and hence 
we obtain {\em the Birman exact sequence} \cite{birman_book} (see also \cite{fm}):
  \begin{equation}
	 1\; \xrightarrow \;  \fg \; \xrightarrow{i} \; 
	 \mcg\; \xrightarrow{j} 
	 \; \M_g\; \xrightarrow \; 1
         \label{Birman_exact}
   \end{equation} 
	         Here we need to take some care about the definition of the mapping $i$.
	         Under our convention, which seems opposite to the convention in \cite{birman_book},
	         that the multiplication of two mapping classes is read from right to left, the original 
	         mapping defined for $i$ by Birman \cite{birman_book} 
	         (``spin map'')  is an anti-homomorphism of group. 
                 To avoid this, we define $i$ as the homomorphism obtained from the 
                 original mapping  by pre-composing the anti-automorphism of $\fg$ defined by the 
                 correspondence $\gamma \mapsto \gamma^{-1}$ for each $\gamma \in \fg$.
		 Explicitly, our mapping $i$ is  described as follows.
 	         For any $\gamma \in \fg$, let $l: [0,1] \to \Sigma_g$ be 
		 its representing loop with $l(0) = l(1) = * \in \Sigma_g$. Considering $l$ 
		 as an isotopy of the marked point $*$, take the isotopy 
	        $H: \Sigma_g \times [0,1] \to \Sigma_g$ extending $l$ and starting with the identity 
	        so that $H(*,t) = l(t)$ for 
	        each $t \in [0,1]$ and the initial homeomorphism $H(\cdot, 0)$ is the identity 
	        of $\Sigma_g$. Denote the ending homeomorphism $H(\cdot, 1)$ 
	         by $h_l$, which is actually a homeomorphism $\Sigma_{g,*} \to \Sigma_{g,*}$. 
	         Then the image $i(\gamma)$ is the isotopy class 
	         of $h_l^{-1}$ in $\M_{g,*}$. For more details, we refer to \cite{fm}.
\par\medskip

Hereafter, for simplicity of notation, 
we identify $\fg$ with its image in $\mcg$ under $i$, and omit the symbol $i$, which would cause no confusion.

\par

The following is a fundamental property of the Birman exact sequence:
\begin{equation} \label{fundamental_property}
 f \cdot \gamma \cdot f^{-1} = f_*(\gamma) \quad \text{for $f \in \mcg$, and $\gamma \in \fg$}
\end{equation}
where the multiplication in the left-hand side is the one in $\mcg$, and 
$f_*$ denotes the natural action of $f$ on $\fg$.
In other words, the natural action of $\mcg$ on $\fg$ coincides with the conjugation action in $\mcg$. 
\par\medskip

Now we obtain the following as a combination of the Dehn--Nielsen theorem and the Birman exact sequence:
\begin{lemma} \label{Lemma1}
 Let $\varphi: \mcg \to G$ be an  arbitrary homomorphism into any group $G$.
Then $\varphi$ is injective if and only if its restriction to $\fg$ is injective.
\end{lemma}
\begin{proof}
First, the only if part is trivial. Suppose next that $\varphi$ is injective on $\fg$.
For any $f \in \ker{\varphi}$, $\gamma \in \fg$, we have
\begin{align*}
 \varphi(f_* (\gamma))  & = \varphi(f \cdot \gamma \cdot f^{-1}) \quad 
                                                    \text{by \eqref{fundamental_property}} \\
                      & = \varphi(f) \varphi(\gamma) \varphi(f^{-1}) = \varphi(\gamma).
\end{align*}
Since $\varphi$ is injective on $\fg$, we  have $f_* (\gamma) = \gamma$
for each $\gamma \in \fg$. Therefore, we have $f = 1$, by the above Dehn--Nielsen theorem. 
This completes the proof.
\end{proof}
\par
%
\subsection{Deformation space} \par

We next consider the problem when a given representation of $\fg$ extends to that of $\mcg$.
To do so, it turns out that the language  of deformation space by Goldman is appropriate, which we 
now recall from \cite{goldman}.
\par

Let $G$ be an arbitrary group. Its automorphism group is denoted by $\Aut{G}$.
We denote by $R_G = \HG$ the set of {\em all} homomorphisms $\fg \to G$. 
The product group $\mcg \times \Aut{G}$ acts on $R_G$ on the left, by pre- and post-composition:
if $f \in \mcg$, and $k \in \Aut{G}$, then the action of $(f, k)$ on $\phi \in R_G$ is given by 
$$(f, k) \cdot \phi  :=  k \circ \phi \circ f_*^{-1}:   \quad
	 \fg\; \xrightarrow{f_*^{-1}} \;  \fg \; \xrightarrow{\phi} \; G\; \xrightarrow{k}  \; G.$$
This action gives rise to actions of both $\mcg$ and $\Aut{G}$ on $R_G$ via the injective homomorphisms
\begin{align*}
 \mcg\; \xrightarrow{{\rm id} \times 1} \; \mcg \times \Aut{G}, \quad & f  \mapsto (f,1); \\
 \Aut{G}\; \xrightarrow{1 \times {\rm id}} \; \mcg \times \Aut{G}, \quad & k \mapsto (1,k),
\end{align*}
respectively.
\par

The {\em deformation space} is the quotient of $R_G$ by the subgroup $\Inn{G}$ of $\Aut{G}$ consisting of 
all the {\em inner automorphisms} of $G$, and is denoted by $X_G = R_G/G$.

Since the actions of $\mcg$ and $\Aut{G}$ on $R_G$ 
are commutative, the action of $\mcg$ on $R_G$ descends to the action on $X_G$. 
On the other hand, in view of  \eqref{fundamental_property}, the action of $\gamma \in \fg$ 
on an element $\phi$ of $R_G$ as an element of $\mcg$ coincides with that of the inner automorphism of $G$ 
induced by $\phi(\gamma)^{-1}$. Hence the normal subgroup $\fg$ of $\mcg$ acts on 
$X_G$ trivially, and therefore  the action of $\mcg$ on $X_G$ descends, via $j$, 
to that of $\M_g$, the mapping class group of the {\em closed} surface 
$\Sigma_g$. In this terminology, a necessary condition for our problem is given as follows.
\begin{lemma} \label{Lemma2}
 If $\phi \in R_G$ extends  to a homomorphism $\mcg \to G$, then 
 its representing class $[\phi] \in X_G$ is a {\em global fixed point} of 
 the $\M_g$-action on $X_G$.
\end{lemma}

\begin{proof}
 Suppose that $\varphi : \mcg \to G$ is a homomorphism extending $\phi \in R_G$.
 We then have,  for every $f \in \mcg$ and  $\gamma \in \fg$,
  \begin{align*}
     (f \cdot \phi)(\gamma) & = \phi(f_{*}^{-1} (\gamma)) = \phi(f^{-1} \cdot \gamma \cdot f) \\
       &  = \varphi(f^{-1} \cdot \gamma \cdot f) = \varphi(f^{-1}) \varphi(\gamma) \varphi(f^{-1})^{-1} \\
       &  = \varphi(f^{-1})\phi(\gamma)\varphi(f^{-1})^{-1}.
   \end{align*}
 Here, for the second equality, we used the formula \eqref{fundamental_property}.
 This shows that the action of $f$ on $\phi$ coincides with that of the inner automorphism of $G$ 
 defined by $\varphi(f)^{-1}$, and hence $f \cdot [\phi] = [\phi]$ in $X_G$.
\end{proof}
\par

\subsection{A global fixed point and a representation of $\mcg$} \par

It seems that the converse of Lemma \ref{Lemma2} is not true, 
partly because of  non-trivial centralizer of $\phi(\fg)$ in $G$. 
However, when $G=\GL{n,K}$, any representation in $R_G$ can be associated with 
another linear representation of $\fg$ so that the associated representation 
always extends to a linear representation of $\mcg$, provided that the initial 
representation corresponds to a global fixed point in $X_G$. 
Furthermore, any faithful representation is to be  associated with a faithful one.

\par

Now, let $G=\GL{n,K}$ with $K$ a field. 
\begin{definition} \label{ap}
 For an arbitrary linear representation $\phi: \fg \to \GL{n,K}$, the finite dimensional vector space 
$V_{\phi}$ is defined as the subspace of $\End{n,K}$ generated by the set 
$\{\phi(\gamma); \gamma \in \fg \}$. 
The {\em associated linear representation of $\phi$}, denoted by 
      $$\ap: \fg \to \GL{V_{\phi}},$$
is defined by  $\ap(\gamma) (x) = \phi(\gamma) \cdot x \cdot \phi(\gamma)^{-1}$ for $\gamma \in \fg$ 
and $x \in V_{\phi}$.
\end{definition}
\par

\begin{remark} \label{remark_ap}
 The well-definedness of $\ap$ follows from the formula 
 $$\ap(\gamma) (\phi(\xi)) = \phi(\gamma \xi \gamma^{-1}) \quad \text{($\gamma$, $\xi \in \fg$).}$$
Obviously, we have
$$\ker{\phi} \subset \ker{\ap} = \{ \gamma \in \fg ; [\gamma, \fg]   \subset \ker{\phi} \}.$$
In particular, since the center of $\fg$ is trivial for $g \geq 2$, the associated representation 
$\ap$ is faithful if and only if $\phi$ is faithful.
\end{remark}

\begin{lemma} \label{Lemma3}
 If $\phi \in R_{\GL{n,K}}$ represents a global fixed point of the $\M_g$-action on $X_{\GL{n,K}}$, 
 then the associated representation $\ap$ extends uniquely to a linear representation of $\mcg$, denoted by 
 $\Psi : \mcg \to \GL{V_{\phi}}$  such that 
 \begin{equation} \label{caps_psi_def}
  \Psi(f) (\phi(\gamma)) = \phi(f_*(\gamma)) \quad \text{for all $f \in \mcg$ and $\gamma \in \fg$.}
 \end{equation}
\end{lemma}
\begin{proof}
 Suppose that $\phi \in R_{\GL{n,K}}$ represents a global fixed point $[\phi] \in X_{\GL{n,K}}$ for 
 the action  of $\M_g$. Then for each $f \in \mcg$, we have
 $ f^{-1} \cdot [\phi] = [\phi]$ in $X_{\GL{n,K}}$, and hence there exists 
an element $x_f \in \GL{n,K}$ such that 
 \begin{equation} \label{x_f}
  \phi(f_*(\gamma)) = x_f \cdot \phi(\gamma) \cdot x_f^{-1} \quad \text{for each $\gamma \in \fg$}
 \end{equation}
Let $a_f: V_{\phi} \to \End{n,K}$ be the linear homomorphism defined by 
$a_f(M) = x_f \cdot M \cdot x_f^{-1}$ for $M \in V_{\phi}$. Since $V_{\phi}$ is generated 
by $\phi(\fg)$, the formula \eqref{x_f} implies that this $a_f$ 
is actually a linear isomorphism $V_{\phi} \to V_{\phi}$, and is independent of the choice of $x_f$.
\par
Now we define $\Psi(f) = a_f$, which determines $\Psi : \mcg \to \GL{V_{\phi}}$ as a mapping of sets.
If $f = \gamma_0 \in \fg$, we may choose $\phi(\gamma_0)$ as $x_f$ because of 
$\eqref{fundamental_property}$. Therefore, this $\Psi$ is an extension of $\ap$.
The uniqueness of $\Psi$ follows from \eqref{x_f} and the fact that 
$V_{\phi}$ is generated by $\fg$.
\par

Finally, we need to check $\Psi(fh) = \Psi(f) \Psi(h)$ for $f$, $h \in \mcg$.
To do so, we only need to check that this equality holds on $\phi(\fg)$, a generating set of $V_\phi$. 
This is straightforward in view of 
\eqref{x_f}, considering the following diagram:
  $$\xymatrix{
     \fg \ar[rrr]^{\phi} \ar[dr]^{h_*} \ar[dd]_{(fh)_*} & & & \phi(\fg)
      \ar[dl]_{\Psi(h)} \ar[dd]^{\Psi(fh)}   \\
     & \fg \ar[r]^{\phi} \ar[ld]^{f_*} & \phi(\fg) \ar[dr]_{\Psi(f)} & \\
     \fg \ar[rrr]_{\phi }& & & \phi(\fg)
   }$$
This completes the proof of Lemma \ref{Lemma3}

\end{proof}
\par

\subsection{Linearity condition for $\mcg$} \par

Now the desired condition for the linearity of $\mcg$ is given as follows.

\begin{theorem} \label{ThmA}
  Let $K$ be a field. Then $\mcg$ is {\em $K$-linear} if and only if there exists a 
 {\em faithful} linear 
 representation $\phi \in R_{\GL{n,K}}$ for {\em some} $n$ which represents a global fixed point of the
 natural action of $\M_g$ on $X_{\GL{n,K}}$. Furthermore, if such a global fixed point exists, then 
 $\mcg$ admits a faithful linear representation over $K$ of {\em dimension at most $n^2$}.
\end{theorem}
\par

\begin{proof}
 If $\mcg$ admits an $n$-dimensional faithful linear representation over $K$, then
 its restriction to $\fg$ is a faithful representation in $R_{\GL{n,K}}$. 
 By Lemma \ref{Lemma2}, this restriction corresponds to 
 a global fixed point of the $\M_g$-action on $X_{\GL{n,K}}$.
 Suppose conversely that  $\phi \in R_{\GL{n,K}}$ is faithful and represents a global fixed 
 point of the $\M_g$-action on $X_{\GL{n,K}}$. Then, the associated representation $\ap$ in  
 Definition \ref{ap} is faithful as we observed in Remark \ref{remark_ap}.
 By Lemma \ref{Lemma3}, $\ap$ extends to the linear representation $\Psi: \mcg \to \GL{V_{\phi}}$.
 This $\Psi$ is faithful by Lemma \ref{Lemma1}. Finally, the dimension of $V_{\phi}$ is at most $n^2$ 
 because $V_{\phi}$ is a subspace of $\End{n,K}$. This completes the proof.
\end{proof}
\par \medskip

\begin{remark} \label{dynamics_results}
 As described  in Goldman \cite{goldman}, the recent dynamical study of the action 
 of mapping class group on deformation space of the fundamental group of the corresponding surface
 has revealed that this action is a complicated mixture of properly discontinuous action and ergodic one.
 Theorem \ref{ThmA} claims that a faithful linear representation of $\mcg$ 
 can lie only in the extreme opposite to the properly discontinuous part in the deformation space. 
 We also note that even the full ergodicity of the action is not enough to prohibit the existence 
 of a {\em single} global  fixed point.  Also in view of the same observation in \cite{goldman} above,
 Lemma \ref{Lemma2}  seems to explain, at least partly, 
 why so few linear representations of $\mcg$ are known so far.
\end{remark}
\par
\begin{remark}
 	It might be interesting to point out here a consequence of Lemma \ref{Lemma2}, 
        for $G=\PSL{2, \R}$.
	A well known source of elements of the deformation space $X_{\PSL{2,\R}}$ is the hyperbolic 
        structures on $\Sigma_g$. 
	Namely, once the orientation of $\Sigma_g$ is fixed, taking the {\em appropriate} holonomy 
        associates to any hyperbolic structure on $\Sigma_g$ an element of $X_{\PSL{2,\R}}$ 
        which is represented by an injective homomorphism $\fg \to \PSL{2,\R}$ with discrete 
        image in $\PSL{2,\R}$. We note that the set of all such elements of $X_{\PSL{2,\R}}$ can be 
       identified with the Teichm\"uller space of $\Sigma_g$.
	Now one can see, for any element of $X_{\PSL{2,\R}}$ so obtained, say $x$, that 
        its representative homomorphism $\fg \to \PSL{2,\R}$ cannot extend to a 
	homomorphism $\M_{g,*} \to \PSL{2,\R}$. In fact, the stabilizer of $x$ is finite, 
        since it is the image in $\M_g$ of the orientation preserving isometry group 
	of the corresponding hyperbolic structure on $\Sigma_g$. (Alternatively, the same result follows
        from the fact that the action of $\M_g$ on the Teichm\"uller space of $\Sigma_g$ is properly 
        discontinuous.) 
	In particular, $x$ is not a global fixed point of the action of $\M_g$ on $X_{\PSL{2,\R}}$, 
	and hence the result follows from Lemma \ref{Lemma2}.
\end{remark}
\begin{remark}[Low dimensional complex linear representations]\label{low_degrees}
 Let $\Sigma_{g,n}^{p}$ be a compact connected orientable surface of genus $g$ with $n \geq 0$ 
 boundary components and $p \geq 0$ marked points. We denote by $\PM$ the {\em pure} mapping 
 class group of $\Sigma_{g,n}^{p}$ 
 which is defined as the group of the isotopy classes of the orientation preserving 
 homeomorphisms of $\Sigma_{g,n}^{p}$ which fix the boundary and the marked points pointwisely.
 Note that $\operatorname{\mathcal{PM}}_{g,0}^{1}$ coincides with $\M_{g,*}$.
 Recently, several results have appeared on the non-triviality of the kernel of low dimensional
 linear representations of $\PM$ over $K = \C$, such as \cite{FLM}, 
 \cite{franks-handel},  \cite{korkmaz1}, and \cite{korkmaz2}. These results apparently give 
 restrictions on the existence of faithful linear representations of $\PM$. 
 Among such restrictions, the best one is given by Korkmaz \cite{korkmaz2} for $g \geq 3$ 
 that any complex linear representation of $\PM$ of dimension less than or equal to 
 $3g-3$ has non-trivial kernel.
 In our context, this result implies that a global fixed point as in Theorem \ref{ThmA} does not 
 exist, for $K = \C$, in the range $g \geq 3$ and $n \leq \sqrt{3g-3}$.
\end{remark}
\par

\begin{remark}
       It seems that there are few existence results on global fixed points of the action of $\M_g$ on the 
       deformation space.
       A remarkable exception is provided by Andersen \cite{andersen}.
       Let $\SU{n}$ denote the special unitary group of degree $n$. 
       In \cite{andersen}, 
       Andersen proved, for each $g \geq 2$ and for infinitely many integers $n$,  
       that there exists an {\em irreducible} representation $\fg \to \SU{n}$ which represents
       a global fixed point of the action of $\M_g$ on $X_{\SU{n}}$.
       Via the inclusion $\SU{n} \hookrightarrow \GL{n,\C}$, such a global fixed point 
       gives rise to  a global fixed point of the $\M_g$-action on $X_{\GL{n,\C}}$, and hence a complex 
       linear representation of $\M_{g,*}$ by Theorem \ref{ThmA}. 
       We remark that all the linear representations of
       $\M_{g,*}$ so constructed have finite image. This follows from the proof of Theorem \ref{ThmA}, 
       together with the fact that each global fixed point constructed in \cite{andersen}
       is represented by a representation $\fg \to \SU{n}$ whose image is finite.
\end{remark}
\par

\subsection{A comparison with the automorphism group of a free group}\par

Let $F_r$ be a free group of rank $r \geq 2$. Since $F_r$ is center free, 
the inner automorphism group of $F_r$ can be 
naturally identified with $F_r$ so that we have an analogue of the Birman exact sequence:
\begin{equation*}
  	 1\; \xrightarrow \; F_r \; \xrightarrow{} \; 
	 \Aut{F_r}\; \xrightarrow{} 
	 \; \Out{F_r} \; \xrightarrow \; 1
         \label{free_Birman_exact}
\end{equation*}
where $\Out{F_r}$ denotes the outer automorphism group of $F_r$. It can be easily checked that
the action of $\Aut{F_r}$ on $F_r$ by conjugation coincides with the natural action.
The same argument above in this section yields the following analogue of  Theorem \ref{ThmA}.
\begin{prop}
   Let $K$ be a field. Then $\Aut{F_r}$ is {\em $K$-linear} if and only if there exists a 
 {\em faithful} linear 
 representation $\phi \in \Hom{F_r, \GL{n,K}}$ for {\em some} $n$ which represents a global fixed point of the
 natural action of $\Out{F_r}$ on $\Hom{F_r, \GL{n,K}}/{\GL{n,K}}$. 
 Furthermore, if such a global fixed point exists, then 
 $\Aut{F_r}$ admits a faithful linear representation over $K$ of {\em dimension at most $n^2$}.
\end{prop}
On the other hand, Formanek--Procesi \cite{fp}
has shown that $\Aut{F_r}$ is {\em not linear if $r \geq 3$}. For the case of $r=2$, the $\C$-linearity 
of $\Aut{F_2}$ was established by Krammer \cite{krammer_B4}.
Therefore, we see:
\begin{cor}
 (1) For $r \geq 3$, and any field $K$, there does not exist a faithful linear representation 
 in $\Hom{F_r, \GL{n,K}}$ which represents a  global fixed point of the $\Out{F_r}$-action 
 on $\Hom{F_r, \GL{n,K}}/\GL{n,K}$.
 \par
 (2) For $r = 2$ and $K=\C$, there exists such a faithful linear representation in $\Hom{F_r, \GL{n,\C}}$ for 
 certain $n$.
\end{cor}

It might be an interesting problem to find out a direct proof of this corollary.

\section{Compact  orientable surface} \label{closed} \par

Let $\Sigma_{g,n}$ be a compact connected orientable surface of genus $g \geq 1$ with 
$n \geq 0$ boundary components.
We denote by $\M_{g,n}$ the mapping class group of $\Sigma_{g,n}$, the group of the isotopy classes of 
the orientation preserving homeomorphisms of $\Sigma_{g,n}$ where 
both the homeomorphisms and the isotopies are assumed to fix the boundary 
pointwise. Note that $\M_{g,n}$ coincides with $\operatorname{\mathcal{PM}}_{g,n}^0$ 
in Remark \ref{low_degrees}.

\subsection{Preliminary} \par
We collect necessary results on $\M_{g,n}$ to derive our linearity condition.
Let $\S = \S(\Sigma_{g,n})$ be the set of the isotopy classes of the essential simple closed curves on 
$\Sigma_{g,n}$. Here a simple closed curve is said to be {\em essential} if it is not homotopic to a point nor 
parallel to any boundary component of $\Sigma_{g,n}$. 
\par
We define a mapping $\iota: \S \to \M_{g,n}$ by sending each $C \in \S$ to the right-handed Dehn twist 
along $C$. The natural action of $\M_{g,n}$ on $\S$ has the effect on Dehn twists as follows:
\begin{equation} \label{twist_conjugation}
  \iota(f(C)) =  f \cdot \iota(C) \cdot f^{-1}  \quad \text{for $f \in \M_{g,n}$ and $C \in \S$.}
\end{equation}
\begin{fact}\label{fact_injective}
 The mapping $\iota$ is {\em injective}. 
\end{fact}
The proof of this fact, which can be found in \cite{fm}, for instance, is  done 
by showing that for two distinct elements $C_1$ and $C_2 \in \S$, the natural actions 
of the Dehn twists along them on $\S$ are different. On the other hand, 
every element of the center of $\M_{g,n}$, denoted by $\ctr{\M_{g,n}}$, acts trivially 
on $\S$ ({\em c.f.} Fact \ref{fact_injective} and the formula \eqref{twist_conjugation}).
Therefore, we can slightly generalize this fact as follows:
\begin{lemma} \label{lemma_center}
 For $C_1$, $C_2 \in \S$, if $C_1 \neq C_2$, then $\iota(C_1) \iota(C_2)^{-1} \notin \ctr{\M_{g,n}}$.
\end{lemma}
\par

We also need the following.
\begin{lemma} \label{trivial_action}
 If the natural action of $f \in \M_{g,n}$ on $\S$ is trivial, 
 then $f$ lies in the center $\ctr{\M_{g,n}}$.
\end{lemma}
\begin{proof}
 Suppose $f \in \M_{g,n}$ and $f(C) = C$ for every $C \in \S$. Then,  by \eqref{twist_conjugation}, 
 we have $f \cdot \iota(C) \cdot f^{-1} = \iota(C)$.  Due to Gervais \cite{gervais}, $\M_{g,n}$ is generated, 
 if $g \geq 2$, by $\iota(\S)$, and  if $g = 1$, by $\iota(\S)$ together with the Dehn twists along 
 the simple closed curves parallel to the boundary components of $\Sigma_{g,n}$. 
 Since the latter type of Dehn twists are obviously contained in  $\ctr{\M_{g,n}}$, 
 $f$ commutes with a generating set of $\M_{g,n}$, and hence lies in $\ctr{\M_{g,n}}$. 
\end{proof}

\subsection{Linearity condition} \par

We now consider to derive a linearity condition for $\M_{g,n}$, up to center.
We first prove the following analogue of Lemma \ref{Lemma1}.
\begin{lemma}\label{injectivity_up_to_center}
 Let $G$ be a group, and $\varphi : \M_{g,n} \to G$ be an arbitrary group homomorphism. 
 Then the composition $\varphi \circ \iota$
 is injective if and only if $\ker{\varphi} \subset \ctr{\M_{g,n}}$.
\end{lemma}
\begin{proof}
 Suppose first that $\varphi \circ \iota$ is injective, and let $f \in \ker{\varphi}$.
 Then, for each $C \in \S$, 
 we have 
 \begin{align*}
  \varphi \circ \iota ( f(C) ) & = \varphi( f \cdot \iota( C ) \cdot f^{-1} ) 
                                      \quad \text{(by \eqref{twist_conjugation})} \\
                   & = \varphi(f) \varphi(\iota(C)) \varphi(f)^{-1} = \varphi \circ \iota(C).  
 \end{align*}
Therefore, $f(C) = C$. Hence, by Lemma \ref{trivial_action}, we have $f \in \ctr{\M_{g,n}}$.
 \par

Next, suppose conversely that $\ker\varphi \subset \ctr{\M_{g,n}}$, and we show that $\varphi \circ \iota$ is 
injective. If $\varphi \circ \iota (C_1) = \varphi \circ \iota (C_2)$ for $C_1$, $C_2 \in \S$, then we have
$\varphi( \iota(C_1) \iota(C_2)^{-1}) = 1$, and hence $\iota(C_1) \iota(C_2)^{-1} \in \ker{\varphi}$, which is
contained in $\ctr{\M_{g,n}}$. Now by Lemma \ref{lemma_center}, we have $C_1 = C_2$. 
This completes the proof.
\end{proof}
\par \medskip

We cannot go further along the same line of the once-punctured case, since $\S$ does not 
have a natural group structure. Instead, we consider, for a field $K$,
the $K$-vector space freely generated by
the set $\S$, and denote it by$K[\S]$. The natural action of $\M_{g,n}$ on $\S$ defines
the structure of an  $\M_{g,n}$-module on $K[\S]$ by linearity. Then the linearity condition for $\M_{g,n}$, 
``up to center'', is given by the following.
\begin{theorem} \label{linearity_up_to_center}
 The mapping class group $\M_{g,n}$ admits a finite dimensional linear representation 
 over $K$ with kernel contained in the center $\ctr{\M_{g,n}}$ if and only if $K[\S]$ has 
 an $\M_{g,n}$-invariant  subspace $V$ such that the dimension of $K[\S]/V$ is finite and 
 the projection $p: K[\S] \to K[\S]/V$ is injective on  the subset $\S$.
\end{theorem}
\par 
\begin{proof}
Suppose first that $\varphi: \M_{g,n}  \to \GL{m,K}$ is a linear representation with 
$\ker{\varphi} \subset \ctr{\M_{g,n}}$. We define the  Adjoint representation 
$\Ad{\varphi} : \M_{g,n} \to \GL{\End{m,K}}$ by
     $$\Ad{\varphi}(f)(X) = \varphi(f) \cdot X \cdot \varphi(f)^{-1} \quad 
       \text{for $f \in \M_{g,n}$ and $X \in \End{m,K}$.}$$
This Adjoint representation defines the structure of an $\M_{g,n}$-module on $\End{m,K}$. 
The mapping $\varphi \circ \iota$ extends 
by linearity to a $K$-linear homomorphism, denoted by $\phi : K[\S] \to \End{m,K}$.
By the formula \eqref{twist_conjugation},  $\phi$ is also $\M_{g,n}$-equivariant. 
Now, take $V = \ker{\phi}$. Then, $V$ is an $\M_{g,n}$-invariant subspace of $K[\S]$, and 
the quotient $K[\S]/V$ is isomorphic to the image of $\phi$, which is therefore of finite dimension.
On the other hand, the restriction of $\phi$ on $\S$ is nothing but $\varphi \circ \iota$, which is 
injective by Lemma \ref{injectivity_up_to_center}. Therefore, the projection $p: K[\S] \to K[\S]/V$ 
is also injective on $\S$. This proves the only if part of the theorem.
\par

Next, suppose conversely that $V$ is an $\M_{g,n}$-invariant subspace of $K[\S]$ such that 
the quotient $K[\S]/V$ is of finite dimension and the projection $p: K[\S] \to K[S]/V$ is injective on 
$\S$. Then, the natural action of $\M_{g,n}$ on $\S$ induces a finite dimensional linear representation 
$\varphi_V : \M_{g,n} \to \GL{K[\S]/V}$. Assume $f \in \ker{\varphi_V}$. Then, for every $C \in \S$,
we have
   $$p(f(C)) = \varphi_V(f)(p(C)) = p(C).$$
Hence, by the injectivity of $p$ on $\S$, we have $f(C) = C$ for every $C \in \S$.
Then, Lemma \ref{trivial_action} implies $f \in \ctr{\M_{g,n}}$.
This completes the proof of Theorem \ref{linearity_up_to_center}.
\end{proof}
\par\medskip

Now, due to Paris--Rolfsen \cite[Theorem 5.6]{paris-rolfsen},
the center $\ctr{\M_{g,n}}$ is completely described as follows: 
\begin{enumerate}
 \item For the case of $n=0$, it holds $\M_{g,n} = \M_g$, and if $g \geq 3$, $\ctr{\M_g}$ is trivial; 
           if $g = 1$, or $2$, then $\ctr{\M_g} \cong \Z/2\Z$ and is generated by the hyperelliptic 
           involution.
 \item For the case of $n=1$ and $g=1$, $\ctr{\M_{1,1}} \cong \Z$  and is generated by the ``half-twist''
       along the simple closed curve parallel to the boundary component.
 \item Otherwise, $\ctr{\M_{g,n}}$ is a free abelian group of rank $n$, and is generated by 
       the Dehn twists along the simple closed curves parallel to the boundary components.
\end{enumerate}
In particular, Theorem \ref{linearity_up_to_center} gives a linearity condition for $\M_{g}$ 
for $g \geq 3$.
\begin{cor}\label{linearity_closed_case}
 Suppose $g \geq 3$. Then $\M_{g}$ is $K$-linear if and only if $K[\S]$ has an 
 $\M_g$-invariant subspace $V$ such that the quotient $K[\S]/V$ is of finite dimension and 
the projection $K[\S] \to K[\S]/V$ is injective on the set $\S$.
\end{cor}
\begin{remark}
 By Paris--Rolfsen \cite[Corollary 4.2]{paris-rolfsen}, 
 $\M_g$ contains $\M_{h,1}$ as a subgroup  if $0 \leq h < g$. Hence the linearity of $\M_g$ for $g \geq 2$ 
 implies the linearity of $\M_{h,1}$  for every $h$ with $h < g$. On the other hand, 
 since the mapping class group $\M_{h,*}$ of the once-punctured surface $\Sigma_{h,*}$ 
 is isomorphic to the quotient 
 of $\M_{h,1}$ by its center, the linearity of $\M_{h,1}$  implies that of $\M_{h,*}$ 
 ({\em c.f.} Brendle--Hamidi-Tehrani \cite[Lemma 3.4]{brendle-hamidi}). 
 As mentioned in Introduction,
 the first $h$ for which the linearity of $\M_{h,*}$ is unknown is $2$. 
 Therefore, the linearity of $\M_{2,*}$ seems an interesting problem, at present.
\end{remark}
\par

\providecommand{\bysame}{\leavevmode\hbox to3em{\hrulefill}\thinspace}
\providecommand{\MR}{\relax\ifhmode\unskip\space\fi MR }
\providecommand{\MRhref}[2]{%
  \href{http://www.ams.org/mathscinet-getitem?mr=#1}{#2}
}
\providecommand{\href}[2]{#2}

\vspace{12pt}
\noindent
\textsc{Yasushi Kasahara \\
Department of Mathematics \\
  Kochi University of Technology \\ Tosayamada,  Kami City, Kochi \\ 
  782-8502 Japan} \\
E-mail: {\tt kasahara.yasushi@kochi-tech.ac.jp}

\end{document}